\documentclass[a4paper, 10pt]{amsart}

\usepackage{amsmath,amssymb, hyperref, color}

\newtheorem{theorem}{Theorem}[section]
\newtheorem{lemma}[theorem]{Lemma}
\newtheorem{corollary}[theorem]{Corollary}
\newtheorem{proposition}[theorem]{Proposition}

\theoremstyle{definition}
\newtheorem{example}[theorem]{Example}

\newcommand{\N}{\mathbb N}
\newcommand{\Z}{\mathbb Z}
\newcommand{\R}{\mathbb R}
\newcommand{\Q}{\mathbb Q}
\newcommand{\F}{\mathbb F}
\newcommand{\C}{\mathbb C}

 \DeclareMathOperator{\ord}{ord}
\DeclareMathOperator{\lcm}{lcm} 

 \DeclareMathOperator{\supp}{supp}

\DeclareMathOperator{\ch}{char}

\renewcommand{\t}{\, | \,}
\renewcommand{\time}{\negthinspace \times \negthinspace}

\newcommand{\be}{\begin{equation}}
\newcommand{\ee}{\end{equation}}

\newcommand{\bnml}{\begin{multline}}
\newcommand{\enml}{\end{multline}}
\newcommand{\buml}{\begin{multline*}}
\newcommand{\euml}{\end{multline*}}

\newcommand{\ber}{\begin{eqnarray}}
\newcommand{\eer}{\end{eqnarray}}

\numberwithin{equation}{section}

\begin{document}

\title{Clean group rings over localizations of rings of integers}

\author{Yuanlin Li }
\address{Department of Mathematics and Statistics,
Brock University, 1812 Sir Isaac Brock Way, St. Catharines, Ontario, Canada L2S 3A1
and 
Faculty of Science, Jiangsu University, Zhenjiang, Jiangsu China
}
\email{yli@brocku.ca
}

\author{Qinghai Zhong}

\address{Institute for Mathematics and Scientific Computing,  University of Graz, NAWI Graz \\
 Heinrichstra{\ss}e 36, 8010 Graz, Austria}
 \email{qinghai.zhong@uni-graz.at}

\thanks{This research was supported in part  by a  Discovery Grant from the Natural Sciences and Engineering Research Council of Canada and  by  the Austrian Science Fund FWF, Project Number P 28864-N35.}



\begin{abstract}
A ring $R$ is said to be clean if each element of $R$ can be written as the sum of a unit and an idempotent. In a recent article (J. Algebra, 405 (2014), 168-178), Immormino and McGoven characterized when the group ring  $\Z_{(p)}[C_n]$ is clean, where $\Z_{(p)}$ is the localization of the integers at the prime $p$. In this paper, we consider a more general setting. Let $K$ be an algebraic number field, $\mathcal O_K$ be its ring of integers, and  $R$ be a localization of $\mathcal O_K$ at some prime ideal. We investigate when $R[G]$ is  clean, where $G$ is a finite abelian group, and obtain  a complete characterization for such a group ring to be clean for the case when $K=\Q(\zeta_n)$ is  a cyclotomic field or $K=\Q(\sqrt{d})$ is  a quadratic field.
\end{abstract}

\maketitle

\bigskip
\section{Introduction}
\medskip

All rings considered here are associative with identity $1\neq 0$. An element of a ring $R$ is called clean if it is the sum of a unit and an idempotent, and a ring $R$ is called clean if each element of $R$ is clean.
Clean rings were introduced and related to exchange rings by Nicholson in 1977 \cite{Nicholson77}  and the study of clean rings has attracted a great deal of attention in recent $2$ decades. For some fundamental properties about clean rings as well as a nice history of clean rings we suggest the interested reader to check the article \cite{McGoven06}.

Let $G$ be a multiplicative group. We denote by $R[G]$ the group ring of $G$ over $R$ which is the set of all formal sums
\[
\alpha=\sum_{g\in G}\alpha_gg\,,
\]
where $\alpha_g\in R$ and the support of $\alpha$, $\supp(\alpha)=\{g\in G\mid \alpha_g\neq 0\}$, is finite.
We let $C_n$ denote the cyclic group of order $n$. Since a homomorphic image of a clean ring is a clean ring, it follows  that it is necessary that $R$ is clean whenever $R[G]$ is.

In this paper, we  investigate the question of when a commutative group ring $R[G]$  over a local ring $R$ is clean. We also study when such a group ring is $*$-clean (see  next section for the definition of $*$-clean rings).
Let $\Z_{(p)}$ denote the localization of the ring $\Z$ of integers at the prime $p$. In \cite{Ha-N01}, the authors  proved that $\Z_{(7)}[C_3]$ is not clean. It then follows that since $\Z_{(p)}$ is a clean ring (as it is local) that $R$ being a commutative clean ring is not sufficient for $R[G]$ to be a clean ring. In a recent paper \cite{Im-Mc14a}, it was shown that $\Z_{(p)}[C_3]$ is clean if and only if $p\not\equiv 1\pmod 3$. More generally, the authors gave a complete characterization of when $\Z_{(p)}[C_n]$ is clean. Note that $\Z_{(p)}$ is a local ring between $\Z$ and $\Q$. In this paper, we consider a more general setting. Let $(R, \mathfrak m)$ be a commutative local ring and we denote $\overline{R}=R/\mathfrak m$. Let $K$ be an algebraic number field, $\mathcal O_K$ be its ring of integers, and  $R$ be a localization of   $\mathcal O_K$ at some prime ideal $\mathfrak p$. We investigate when $R[G]$ is  clean, where $G$ is a finite abelian group, and provide a complete characterization for such a group ring to be clean  for the case when $K=\Q(\zeta_n)$ is  a cyclotomic field or $K=\Q(\sqrt{d})$ is  a quadratic field. Our main results are as follows.

\begin{theorem}\label{main1}
Let $K=\Q(\zeta_n)$ be a cyclotomic field for some $n\in \N$, $\mathcal O=\Z[\zeta_n]$ its rings of integers, $\mathfrak p\subset \mathcal O$ a nonzero prime ideal, and $G$ a finite abelian group.
Let $p$ be the prime with $p\Z=\mathfrak p\cap \Z$,
let  $n_0$ be  the maximal positive divisor of $n$ with $p\nmid n_0$, let $n_1$ be the maximal divisor of $\exp(G)$ with  $p\nmid n_1$ and $\gcd(n_1, n_0)=1$, and let $m'$ be the maximal divisor of $\frac{\lcm(\exp(G), n_0)}{n_0n_1}$ with $p\nmid m'$.

 Then the group ring $\mathcal O_{\mathfrak p}[G]$ is clean
 if and only if $\ord_{n_1}p=\varphi(n_1)$, $\ord_{n_0m'}p=m'\ord_{n_0}p$, and $\gcd(\ord_{n_1}p,\ord_{n_0m'}p)=1$. In particular, if $\exp(G)$ is a divisor of $n$, then $\mathcal O_{\mathfrak p}[G]$ is clean.
\end{theorem}

Note that if $n=1$ (i.e. $K=\Q$), then $n_0=1$ and $m'=1$. Therefore Theorem \ref{main1} implies the following corollary which is exactly the main result of \cite[Theorem 3.3]{Im-Mc14a}.

\begin{corollary}
Let $G$ be a finite abelian group, let $p$ be a prime number, and  let $n_1$ be the maximal divisor of $\exp(G)$ with  $p\nmid n_1$. Then
$\Z_{(p)}[G]$ is clean
 if and only if $\ord_{n_1}p=\varphi(n_1)$(i.e. $p$ is a primitive root of $n_1$). 
\end{corollary}

\begin{theorem}\label{t3}
Let $K=\Q(\sqrt{d})$ be a quadratic field for some non-zero square-free integer $d\neq 1$, $\mathcal O$ its rings of integers, $\mathfrak p\subset \mathcal O$ a nonzero prime ideal, and $G$ a finite abelian group.  Let $p$ be the prime with $p\Z=\mathfrak p\cap \Z$, let $\Delta$ be the  discriminant of $K$, and let $n$ be the maximal positive divisor of $\exp(G)$ with $p\nmid n$. 
 \begin{enumerate}
 \item  If $\Delta\nmid n$, then $\mathcal O_{\mathfrak p}[G]$ is clean if and only if one of the following holds
  \begin{enumerate}
  \item $p=2$ is a primitive root of unity of $n$ and $\Delta\not\equiv 5\pmod 8$;

   \item  $p\neq 2$ is a primitive root of unity of $n$ and   $\big(\frac{\Delta}{p}\big)=1$ or $0$;

    \item $n=2$, $p\neq 2$, and   $\big(\frac{\Delta}{p}\big)=-1$, where $\big(\frac{\Delta}{p}\big)$ is the Legendre symbol.
  \end{enumerate}

 \item If $\Delta\t n$ and $d\equiv 2,3\pmod 4$, then $\mathcal O_{\mathfrak p}[G]$ is clean if and only if
 $|d|$ is a prime, $n=4|d|^{\ell}$ with $\ell\in \N$,
  $p\equiv 3\pmod 4$, $\big(\frac{\Delta}{p}\big)=1$, and $\ord_np=\varphi(n)/2$.

 \item If $\Delta\t n$ and $d\equiv 1\pmod 4$, then $\mathcal O_{\mathfrak p}[G]$ is clean if and only if  one of the following holds
  \begin{enumerate}
  \item $|d|$ is a prime,  $n=|d|^{\ell}$ or $2|d|^{\ell}$ for some $\ell\in \N$, and  either $\ord_np=\frac{2\varphi(n)}{3+\big(\frac{d}{p}\big)}$, or $\ord_np=\varphi(n)/2$  with $\big(\frac{d}{p}\big)=-1$ and $d<0$, where $\big(\frac{d}{p}\big)$ is the Legendre symbol.

\item $|d|=q$ is a prime, $n=4q^{\ell}$ with $q\equiv 3\pmod 4$ and $\ell\in \N$, $\big(\frac{d}{p}\big)=1$, 
 $p\equiv 3\pmod 4$, and $\ord_{q^{\ell}}p=q^{\ell-1}(q-1)/2$.

	\item $|d|=q_1$ is a prime, $n=q_1^{\ell_1}q_2^{\ell_2}$ or $2q_1^{\ell_1}q_2^{\ell_2}$ with $q_1\equiv 3\pmod 4$, $\ell_1,\ell_2\in \N$, and $q_2$ is another odd prime,  $\big(\frac{d}{p}\big)=1$, 
	 $p$ is a primitive root of unity of $q_2^{\ell_2}$, $\ord_{q_1^{\ell_1}}p=q_1^{\ell_1-1}(q_1-1)/2$, and $\gcd(q_1^{\ell_1-1}(q_1-1)/2, q_2^{\ell_2-1}(q_2-1))=1$.

  \item $|d|=q_1q_2$ is a product of two distinct primes, $n=q_1^{\ell_1}q_2^{\ell_2}$ or $2q_1^{\ell_1}q_2^{\ell_2}$ with $\ell_1,\ell_2\in \N$,  $\big(\frac{d}{p}\big)=1$,  $p$ is a primitive root of unity of $q_1^{\ell_1}$ and $q_2^{\ell_2}$, and $\gcd(q_1^{\ell_1-1}(q_1-1)/2, q_2^{\ell_2-1}(q_2-1)/2)=1$.

  \end{enumerate}
 \end{enumerate}

\end{theorem}

In Section 2, we collect some necessary knowledge of the structure of the group of units $(\Z/m\Z)^{\times}$ and field extension. Furthermore, we give some general characterization theorems for clean and $*$-clean group rings. In Section 3, we deal with group rings over local subrings of cyclotomic fields and provide a proof of Theorem \ref{main1}.  In Section 4, we consider group rings over local subrings of quadratic fields and give a proof of Theorem \ref{t3}.

\bigskip
\section{Preliminaries}
\medskip

For a finite abelian group $G$, we denote by $\exp(G)$ the exponent of $G$. 
We denote by $\N$ the set of all positive integers and $\N_0=\N\cup\{0\}$.
For  $n\in \N$,  we denote by $\varphi(n)$ the Euler function.
Let $n\in \N$ and let $n=p_1^{k_1}\ldots p_s^{k_s}$ be its prime factorization, where $s, k_1,\ldots, k_s\in \N$ and $p_1,\ldots, p_s$ are pair-wise distinct primes. It is well-known that
 \begin{align*}
 &\varphi(n)=\prod_{i=1}^s\varphi(p_i^{k_i})=\prod_{i=1}^sp_i^{k_i-1}(p_i-1)\\
\text{ and }\quad &(\Z/n\Z)^{\times}\cong (\Z/p_1^{k_1}\Z)^{\times}\times\ldots\time(\Z/p_s^{k_s}\Z)^{\times}\,.
\end{align*}
Furthermore,
\begin{align*}
&(\Z/p_i^k\Z)^{\times}\cong C_{p_i^{k-1}(p_i-1)}& \quad &\text{ where } p_i\ge 3,\\
&(\Z/2^{\ell}\Z)^{\times}=\langle-1\rangle\times\langle5\rangle\cong C_2\oplus\C_{2^{\ell-2}} &\quad &\text{ where }\ell\ge 3,\\
\text{ and }\quad &(\Z/4\Z)^{\times}\cong C_2\,.&&
\end{align*}

For every $m\in \N$ with $\gcd(m,n)=1$, we denote by $\ord_nm=\ord_{(\Z/n\Z)^{\times}}m$ the multiplicative order of $m$ modulo $n$.
If $\ord_nm=\varphi(n)$, we say $m$ is a primitive root of $n$ and $n$ has a primitive root if and only if $n=2$, $4$, $q^{\ell}$, or $2q^{\ell}$, where $q$ is an odd prime and $\ell\in \N$. Let $n_1\in \N$ be another integer with $\gcd(n_1,m)=1$. Then
\begin{align*}
&\ord_nm\le \ord_{nn_1}m\le n_1\ord_nm\,,\\
\text{ and }\quad &\lcm(\ord_nm,\ord_{n_1}m)=\ord_{\lcm(n,n_1)}m\,.
\end{align*}

Let $\zeta_n$ be an $n$th primitive root of unity over $\Q$. Then $[\Q(\zeta_n):\Q]=\varphi(n)$. Let $m$ be another positive integer. Then
\begin{align*}
&\Q(\zeta_n)\cap \Q(\zeta_m)=\Q(\zeta_{\gcd(n,m)})\\
\text{ and }\quad  &\Q(\zeta_n)(\zeta_m)=\Q(\zeta_{\lcm(n,m)})\,.
\end{align*}

\medskip

Let $(R, \mathfrak m)$ be  a commutative local ring and we denote $\overline{R}=R/\mathfrak m$. Then $\overline{R}$ is a field and we denote by $\ch \overline{R}$ the characteristic  of $\overline{R}$. For any polynomial $f(x)=a_nx^n+\ldots+a_0\in R[x]$, we denote $\overline{f(x)}=\overline{a_n}x^n+\ldots+\overline{a_0}\in \overline{R}[x]$, where $\overline{a_i}=a_i+\mathfrak m$ for all $i\in \{0,\ldots,n\}$.

Let $R$ be a ring $R$  and let $G$ be  a  multiplicative group. Then the group ring of $G$ over $R$ is the ring $R[G]$ of all formal sums
\[
\alpha=\sum_{g\in G}\alpha_gg\,,
\]
where $\alpha_g\in R$ and the support of $\alpha$, $\supp(\alpha)=\{g\in G\mid \alpha_g\neq 0\}$, is finite. Addition is defined componentwise and multiplication is defined by the following way: for $\alpha, \beta\in R[G]$,
\[
\alpha\beta=\left(\sum_{g\in G}\alpha_gg\right)\left(\sum_{h\in G}\beta_hh\right)=\sum_{g, h\in G}\alpha_g\beta_h(gh)\,.
\]
For more information on the group ring, we refer \cite{MS} as a reference. 

\begin{theorem}\label{t1}

Let $(R, \mathfrak m)$ be a  commutative noetherian local ring with $\ch \overline{R}=p\ge 0$, let $G$ be a finite abelian group, and let $n$ be the maximal divisor of $\exp(G)$ with $p\nmid n$. Then $R[G]$ is clean if and only if  each monic factor of $x^n-1$ in $\overline{R}[x]$ can be lifted to a monic factor of $x^n-1$ in $R[x]$.
\end{theorem}
\begin{proof}
This  follows from  \cite[Proposition 2.1]{Im-Mc14a} and \cite[Theorem 5.8]{Wo74a}.
\end{proof}

\medskip
Let $K$ be an algebraic number field, $\mathcal O$ its rings of integers, and $\mathfrak p\subset \mathcal O$ a nonzero prime ideal. Then there exists a prime $p$ such that $\mathfrak p\cap \Z=p\Z$ and the localization $\mathcal O_{\mathfrak p}$ is a discrete valuation ring, which implies that $\mathcal O_{\mathfrak p}[x]$ is a UFD (Unique factorization domain).
Furthermore, the norm $N(\mathfrak p)=|\mathcal O/\mathfrak p|=|\overline{\mathcal O_{\mathfrak p}}|$ is a prime power of $p$.

Let $\F_q$ be a finite field, where $q$ is a power of some prime $p$ and let $\overline{\zeta_n}$ be the $n$th primitive root of unity over $\F_p$ with $\gcd(n,q)=1$. Then $[\F_q(\overline{\zeta_n}):\F_q]=\ord_nq$.
Let $F$ be an arbitrary field and let $f(x)$ be a polynomial of $F[x]$. If $\alpha$ is a root of $f(x)$, then $[F(\alpha):F]=\deg(f(x))$ if and only if $f(x)$ is irreducible in $F[x]$.

\begin{theorem}\label{t2}
Let $K$ be an algebraic number field, $\mathcal O$ its rings of integers, $\mathfrak p\subset \mathcal O$ a nonzero prime ideal, and $G$ a finite abelian group. Then the group ring $\mathcal O_{\mathfrak p}[G]$ is clean if and only if $[K[\zeta_m]\colon K]=\ord_m( N(\mathfrak p))$ for every positive divisor $m$ of $\exp(G)$ with $p\nmid m$,
 where  $\zeta_m$ is  an  $m$th primitive root of unity over $\Q$.
\end{theorem}

\begin{proof}
Let $n$ be the maximal divisor of $\exp(G)$ with $p\nmid \exp(G)$. Since $\mathcal O_{\mathfrak p}[x]$ is a UFD, we suppose that
 $x^n-1=f_1(x)\cdot \ldots\cdot f_s(x)$ , where $s\in \N$ and $f_1(x), \ldots, f_s(x)$ are monic irreducible polynomials in $\mathcal O_{\mathfrak p}[x]$.
 Then the Generalized Gauss' Primitive Polynomial Lemma implies that  $f_1(x), \ldots, f_s(x)$ are also monic irreducible polynomials in $K[x]$. For every positive divisor $m$ of $n$,  let $\varPhi_m(x)$ be the $m$th cyclotomic polynomial. Then $\varPhi_m(x)\in \Z[x]\subset \mathcal O_{\mathfrak p}[x]$ and $$x^n-1=\prod_{1<m\t n}\varPhi_m(x)=f_1(x)\cdot \ldots\cdot f_s(x)\,.$$
\medskip

1. Suppose $\mathcal O_{\mathfrak p}[G]$ is clean. Let $f(x)$ be a monic irreducible factor of $x^n-1$ in $\mathcal O_{\mathfrak p}[x]$ and let $\mathfrak h(x)$ be a monic irreducible factor of $\overline{f(x)}$ in $\overline {\mathcal O_{\mathfrak p}}[x]$. By Theorem \ref{t1}, there exists a monic irreducible factor $h(x)$ of $x^n-1$ in $\mathcal O_{\mathfrak p}[x]$ such that $\overline{h(x)}= \mathfrak h(x)$. If $h(x)\neq f(x)$, it follows by $\mathcal O_{\mathfrak p}[x]$ is a UFD that $f(x)h(x)$ is a monic factor of $x^n-1$ in $\mathcal O_{\mathfrak p}[x]$ and hence ${\overline{h(x)}}^2$ is a monic factor of $x^n-1$ in $\overline {\mathcal O_{\mathfrak p}}[x]$. Since $\gcd(n,p)=1$, we obtain $x^n-1\in \overline{\mathcal O_{\mathfrak p}}[x]$ has no multiple root in any extension of $\overline{\mathcal O_{\mathfrak p}}$, a contradiction. Therefore $h(x)=f(x)$ and hence $\overline{f(x)}=\overline{h(x)}= \mathfrak h(x)$ is irreducible in $\overline{\mathcal O_{\mathfrak p}}[x]$.

 Let $m$ be a positive divisor of $n$. It  follows from the fact that $\mathcal O_{\mathfrak p}[x]$ is a UFD that
 there exists $i\in [1,s]$ such that $f_i(x)$ divides  $\varPhi_m(x)$ in $\mathcal O_{\mathfrak p}[x]$.
Thus every root of $f_i(x)$ is a $m$th primitive root of unity in $K$ and hence $[K(\zeta_m):K]=\deg(f_i(x))=\deg(\overline{f_i(x)})$.  Since  $\overline{f_i(x)}$ is irreducible in $\overline{\mathcal O_{\mathfrak p}}[x]$ and every root of $\overline{f_i(x)}$ is a $m$th primitive root of unity in $\overline{\mathcal O_{\mathfrak p}}$, we have $\deg(\overline{f_i(x)})=[\overline {\mathcal O_{\mathfrak p}}(\overline{\zeta}_m):\overline{\mathcal O_{\mathfrak p}}]=\ord_m N(\mathfrak p)$, where $\overline{\zeta}_m$ is a $m$th primitive root of unity over $\F_p$.
Therefore $[K(\zeta_m):K]=\ord_m N(\mathfrak p)$.

\medskip
2. Conversely, suppose $[K(\zeta_m):K]=\ord_m N(\mathfrak p)$ for every divisor $m$ of $n$.
 Let $i\in [1,s]$. Then $f_i(x)$ is a factor of some $m$th cyclotomic polynomial $\varPhi_m(x)$ with $m\t n$. Since $f_i(x)$ is irreducible in $K[x]$, we have $\deg(f_i(x))=[K(\zeta_m):K]$ and hence
 $$\deg(\overline{f_i(x)})=\deg(f_i(x))=[K(\zeta_m):K]=\ord_m N(\mathfrak p)=|\overline {\mathcal O_{\mathfrak p}}(\overline{\zeta}_m):\overline{\mathcal O_{\mathfrak p}}|\,.$$
 Therefore  $\overline{f_i(x)}$ is irreducible in $\overline{\mathcal O_{\mathfrak p}}[x]$ and
 $$x^n-1=\overline{f_1(x)}\cdot \ldots\cdot \overline{f_s(x)}\in \overline {\mathcal O_{\mathfrak p}}[x]\,.$$
Let $\mathfrak h(x)$ be a monic  factor of $x^n-1\in \overline {\mathcal O_{\mathfrak p}}[x]$. Since $\overline {\mathcal O_{\mathfrak p}}[x]$ is a UFD, there exists  a subset $I\subset [1,s]$ such that $\mathfrak h(x)=\prod_{i\in I}\overline{f_i(x)}$ and hence $\overline{\prod_{i\in I}f_i(x)}=\mathfrak h(x)$. Therefore every monoic factor of $x^n-1\in \overline{\mathcal O_{\mathfrak p}}[x]$ can be lifted to a monioc factor of $x^n-1\in \mathcal O_{\mathfrak p}[x]$. It follows from   Theorem \ref{t1} that $\mathcal O_{\mathfrak p}[G]$ is clean.
\end{proof}

\medskip
A ring $R$ is called a $*$-ring if there exists an operation $*: R\rightarrow R$ such that $(x+y)^*=x^*+y^*$, $(xy)^*=x^*y^*$, and $(x^*)^*=x$ for all $x,y\in R$. An element $p\in R$ is said to be a projection if $p^*=p=p^2$ and a $*$-ring $R$ is said to be a $*$-clean ring if every element of $R$ is the sum of a unit and a projection. A commutative $*$-ring is $*$-clean if and only if it is clean and every idempotent is a projection (\cite[Theorem 2.2]{Li-Zhou11}).
Let $G$ be an abelian group.
 With the classical involution
\begin{align*}
* :R[G]&\rightarrow R[G], \mbox{given by }\\
(\sum a_gg)^*&= \sum a_gg^{-1}\,,
\end{align*}
the group ring $R[G]$ is a $*$-ring. The question of when a group ring $R[G]$ is $*$-clean has been recently studied by several authors and many interesting results were obtained (see, for examples, \cite{Gao-Ch-Li15,
Ha-Re-Zh17a, Huang-Li-Tang16, Huang-Li-Yuan16, Li-Zhou11, Li-Par-Yuan15},  for some recent developments). Next we provide a characterization for $\mathcal O_{\mathfrak p}[G]$ to be  $*$-clean.

\begin{theorem}\label{t6}
Let $K$ be an algebraic number field, $\mathcal O$ its rings of integers, $\mathfrak p\subset \mathcal O$ a nonzero prime ideal with $p\Z=\mathfrak p\cap \Z$, and $G$ a finite abelian group with $p\nmid \exp(G)$.
If the group ring $\mathcal O_{\mathfrak p}[G]$ is clean,
then  $\mathcal O_{\mathfrak p}[G]$ is $*$-clean if and only if $K[G]$ is $*$-clean.
\end{theorem}

\begin{proof}
Let $\mathcal O_{\mathfrak p}[G]$ be clean. Suppose $K[G]$ is $*$-clean. Since every idempotent of $\mathcal O_{\mathfrak p}[G]$ is an idempotent of $K[G]$, thus every idempotent of $\mathcal O_{\mathfrak p}[G]$ is a projective. It follows that  $\mathcal O_{\mathfrak p}[G]$ is $*$-clean.

Suppose $\mathcal O_{\mathfrak p}[G]$ is $*$-clean. Let $\mathcal O_{K(\zeta_{\exp(G)})/\Q}$ be the ring of integers of $K(\zeta_{\exp(G)})$ and let  $I$ be  a prime ideal of $\mathcal O_{K(\zeta_{\exp(G)})/\Q}$ with $I\cap \mathcal O=\mathfrak p$.
By \cite[The beginning of Section 5]{Ha-Re-Zh17a} and $p\nmid \exp(G)$, there is a complete family of orthogonal idempotents of $K(\zeta_{\exp(G)})[G]$ which lies in $(\mathcal O_{K(\zeta_{\exp(G)})/\Q})_I[G]$. It follows from   \cite[Lemma 4.3]{Ha-Re-Zh17a} that every idempotent of $K[G]$ lies in $(\mathcal O_{K(\zeta_{\exp(G)})/\Q})_I[G]\cap K[G]=\mathcal O_{\mathfrak p}[G]$. Since every idempotent of $\mathcal O_{\mathfrak p}[G]$ is a projection, we obtain every idempotent of $K[G]$ is a projection. Note that $K[G]$ is clean. Thus $K[G]$ is $*$-clean.
\end{proof}

\bigskip
\section{Group rings over local subrings of cyclotomic fields}
\medskip

In this section, we investigate when a group ring over a local subring of a cyclotomic field is clean and provide a proof for our main theorem \ref{main1}. We also characterize when such a group ring is $*$-clean. We start with  the following  lemma which we will use without further mention.

\begin{lemma}\label{l1}
Let $K=\Q(\zeta_n)$ be a cyclotomic field for some $n\in \N$, $\mathcal O=\Z[\zeta_n]$ its rings of integers, and $\mathfrak p\subset \mathcal O$ a nonzero prime ideal with $\mathfrak p\cap \Z=p\Z$ for some prime $p$. Suppose  $n=p^un_0$ with $p\nmid n_0$. Then $ N(\mathfrak p)=p^{\ord_{n_0}p}$.
\end{lemma}

\begin{proof}
This follows by \cite[VI.1.12 and VI.1.15]{Fr-Ta92}.
\end{proof}

%

\medskip
\begin{proof}[{\bf Proof of Theorem \ref{main1}}]
Let $n_2$ be the maximal divisor of $\exp(G)$ such that $p\nmid n_2$ and
let $n_3=\frac{n_2}{n_1}$. Since
$$m'=\frac{\lcm(n_2, n_0)}{n_0n_1}=\frac{\lcm(n_3, n_0)}{n_0}\,,$$
we have $\lcm(n_3, n_0)=n_0m'$ and $\lcm(n_2, n_0)=n_0m'n_1$.

Let $m$ be a divisor of $n_2$. Then $$[\Q(\zeta_{n})(\zeta_m):\Q(\zeta_{n})]=[\Q(\zeta_{\lcm(n,m)}:\Q(\zeta_{n})]=\frac{\varphi(\lcm(n,m))}{\varphi(n)}=\frac{\varphi(\lcm(n_0,m))}{\varphi(n_0)}$$ and $$\ord_m N(\mathfrak p)=\ord_mp^{\ord_{n_0}p}=\frac{\ord_mp}{\gcd(\ord_mp, \ord_{n_0}p)}=\frac{\lcm(\ord_{n_0}p, \ord_mp)}{\ord_{n_0}p}=\frac{\ord_{\lcm(n_0,m)}p}{\ord_{n_0}p}\,.$$

Therefore by Theorem \ref{t2} that  $R[G]$ is clean if and only if
$$
\text{for every divisor $m$ of $n_2$,\quad  we have }\quad  \frac{\varphi(\lcm(n_0,m))}{\varphi(n_0)}=\frac{\ord_{\lcm(n_0,m)}p}{\ord_{n_0}p}\,.$$

\medskip
1. We first suppose that  $R[G]$ is clean. Since $n_1$, $n_3$ and $n_2$ are  divisors of $n_2$, we obtain
 \begin{align*}
 &\frac{\varphi(\lcm(n_0,n_1))}{\varphi(n_0)}=\frac{\ord_{\lcm(n_0,n_1)}p}{\ord_{n_0}p}\,,\\
 &\frac{\varphi(\lcm(n_0, n_3))}{\varphi(n_0)}=\frac{\ord_{\lcm(n_0,n_3)}p}{\ord_{n_0}p}\,,\\
\text{ and }\quad & \frac{\varphi(\lcm(n_0, n_2))}{\varphi(n_0)}=\frac{\ord_{\lcm(n_0,n_2)}p}{\ord_{n_0}p}\,.
 \end{align*}

Since $\gcd(n_1,n_0)=1$, we obtain \begin{align*}
\varphi(n_1)&=\frac{\varphi(\lcm(n_0,n_1))}{\varphi(n_0)}=\frac{\ord_{\lcm(n_0,n_1)}p}{\ord_{n_0}p}=\frac{\lcm(\ord_{n_0}p, \ord_{n_1}p)}{\ord_{n_0}p}\le\frac{\ord_{n_1}p\ord_{n_0}p}{\ord_{n_0}p}
=\ord_{n_1}p\le \varphi(n_1)\,.
\end{align*}
Then $\ord_{n_1}p=\varphi(n_1)$.

Since
\begin{align*}
m'=\frac{\varphi(n_0m')}{\varphi(n_0)}=\frac{\varphi(\lcm(n_0, n_3))}{\varphi(n_0)}=\frac{\ord_{\lcm(n_0,n_3)}p}{\ord_{n_0}p}=\frac{\ord_{n_0m'}p}{\ord_{n_0}p}\le \frac{m'\ord_{n_0}p}{\ord_{n_0}p} =m'\,,
\end{align*}
we obtain $\ord_{n_0m'}p=m'\ord_{n_0}p$.

Since \begin{align*}
m'\varphi(n_1)&=\frac{\varphi(n_0m'n_1)}{\varphi(n_0)}=\frac{\varphi(\lcm(n_0, n_2))}{\varphi(n_0)}=\frac{\ord_{\lcm(n_0,n_2)}p}{\ord_{n_0}p}=\frac{\ord_{n_0m'n_1}p}{\ord_{n_0}p}\\
&=\frac{\lcm(\ord_{n_0m'}p, \ord_{n_1}p)}{\ord_{n_0}p}\le \frac{\ord_{n_0m'}p\ord_{n_1}p}{\ord_{n_0}p}=\frac{m'\ord_{n_0}p\varphi(n_1)}{\ord_{n_0}p}=m'\varphi(n_1)\,,
\end{align*}
we obtain $\lcm(\ord_{n_0m'}p, \ord_{n_1}p)=\ord_{n_0m'}p\ord_{n_1}p$. Thus $\gcd(\ord_{n_0m'}p, \ord_{n_1}p)=1$.

\medskip
2.  Conversely, suppose that
$\ord_{n_1}p=\varphi(n_1)$, $\ord_{n_0m'}p=m'\ord_{n_0}p$, and $\gcd(\ord_{n_1}p,\ord_{n_0m'}p)=1$.
Then for every $m\t n_2$, we let $m_1=\gcd(m,n_1)$ and $m_2=\frac{\lcm(m/m_1, n_0)}{n_0}$. Then $\ord_{m_1}p=\varphi(m_1)$. It follows by $n_0m_2\t n_0m'$ that $\gcd(\ord_{m_1}p, \ord_{n_0m_2}p)=1$ and
$$\ord_{n_0m'}p=\ord_{n_0m_2\frac{m'}{m_2}}p\le \frac{m'}{m_2}\ord_{n_0m_2}p\le \frac{m'}{m_2}m_2\ord_{n_0}p=m'\ord_{n_0}p=\ord_{n_0m'}p\,.$$
Thus $\ord_{n_0m_2}p= m_2\ord_{n_0}p$.
It follows that
\begin{align*}
\frac{\ord_{\lcm(n_0,m)}p}{\ord_{n_0}p}&=\frac{\ord_{m_1n_0m_2}p}{\ord_{n_0}p}=
\frac{\lcm(\ord_{m_1}p, \ord_{n_0m_2}p)}{\ord_{n_0}p}\\
&=
\frac{\ord_{m_1}p\ord_{n_0m_2}p}{\ord_{n_0}p}=\frac{m_2\ord_{m_1}p\ord_{n_0}p}{\ord_{n_0}p}\\
&= m_2\ord_{m_1}p=m_2\varphi(m_1)\\
&=\frac{\varphi(m_1n_0m_2)}{\varphi(n_0)}=\frac{\varphi(\lcm(n_0,m))}{\varphi(n_0)}\,.
\end{align*}
Therefore $R[G]$ is clean.

\medskip
3. In particular, if $\exp(G)$ is a divisor of $n$, then $n_1=m'=1$ and hence $\mathcal O_{\mathfrak p}[G]$ is clean.
\end{proof}

Next we characterize when a group ring of a finite abelian group over a local ring $\mathcal O_{\mathfrak p} $ is $*$-clean. We need the following two propositions.

\begin{proposition} \label{p1}
Let $m, n\in \N$. Then $\Q(\zeta_m+\zeta_m^{-1})(\zeta_n)=\Q(\zeta_m)(\zeta_n)$ if and only if $\gcd(m,n)\ge 3$ or $m\le 2$.
\end{proposition}
\begin{proof} If $m\le 2$, then it is obvious that $\Q(\zeta_m+\zeta_m^{-1})(\zeta_n)=\Q(\zeta_m)(\zeta_n)$. We suppose $m\ge 3$.

Let $K=\Q(\zeta_m+\zeta_m^{-1})$. Then $K\subset K(\zeta_n)\subset \Q(\zeta_m)(\zeta_n)=\Q(\zeta_{\lcm(n,m)})$. Thus $K(\zeta_n)= \Q(\zeta_m)(\zeta_n)$ if and only if
$[K(\zeta_n):K]=[\Q(\zeta_{\lcm(n,m)}):K]$.

Since $[\Q(\zeta_{\lcm(n,m)}):K]=\frac{[\Q(\zeta_{\lcm(n,m)}):\Q]}{[K:\Q]}=\frac{2\varphi(\lcm(m,n))}{\varphi(m)}=\frac{2\varphi(n)}{\varphi(\gcd(m,n))}$ and
\begin{align*}
[K(\zeta_n):K]=[\Q(\zeta_n):K\cap \Q(\zeta_n)]&=[\Q(\zeta_n):\Q(\zeta_{\gcd(m,n)})][\Q(\zeta_{\gcd(m,n)}):K\cap \Q(\zeta_n)]\\
&=\frac{\varphi(n)}{\varphi(\gcd(m,n))}[\Q(\zeta_{\gcd(m,n)}):K\cap \Q(\zeta_{\gcd(m,n)})]\\
&=\frac{\varphi(n)}{\varphi(\gcd(m,n))}[K(\zeta_{\gcd(m,n)}):K]\,,
\end{align*}
we obtain $|K(\zeta_{\gcd(m,n)}):K|\le 2$. Moreover,
 $K(\zeta_n)= \Q(\zeta_m)(\zeta_n)$ if and only if $|K(\zeta_{\gcd(m,n)}):K|=2$ if and only if $\zeta_{\gcd(m,n)}\not\in K$.

If $\gcd(m,n)\ge 3$, then $\zeta_{\gcd(m,n)}\not\in \R$ and hence $\zeta_{\gcd(m,n)}\not\in K\subset \R$. It follows that $K(\zeta_n)= \Q(\zeta_m)(\zeta_n)$.

If $\gcd(m,n)\le 2$, then  $\zeta_{\gcd(m,n)}\in \Q \subset K$ and hence $K(\zeta_n)\neq \Q(\zeta_m)(\zeta_n)$.
\end{proof}

\begin{proposition}\label{p2}
Let $G$ be a finite abelian group and let $n\in \N$. Then $\Q(\zeta_n)[G]$ is $*$-clean if and only if $\exp(G)\ge3 $ and $\gcd(\exp(G),n)\le 2$.
\end{proposition}

\begin{proof}

This follows from   \cite[Theorem 1.2]{Ha-Re-Zh17a} and Proposition \ref{p1}.
\end{proof}

\begin{theorem}\label{main2}
Let $K=\Q(\zeta_n)$ be a cyclotomic field for some $n\in \N$, $\mathcal O=\Z[\zeta_n]$ its rings of integers, $\mathfrak p\subset \mathcal O$ a nonzero prime ideal with $p\Z=\mathfrak p\cap \Z$, where $p$ is a prime, and $G$ a finite abelian group with $p\nmid \exp(G)$.
let  $n_0$ be  the maximal  divisor of $n$ with $p\nmid n_0$ and let $n_1$ be the maximal divisor of $\exp(G)$ with $\gcd(n_1, n_0)=1$.
 Then the group ring $\mathcal O_{\mathfrak p}[G]$ is $*$-clean
 if and only if $\ord_{n_1}p=\varphi(n_1)$, $3\le\exp(G)\le 2n_1$, and $\gcd(\ord_{n_1}p,\ord_{n_0}p)=1$.
\end{theorem}

\begin{proof}
1. Suppose  $\ord_{n_1}p=\varphi(n_1)$, $3\le\exp(G)\le 2n_1$, and $\gcd(\varphi(n_1), \ord_{n_0}p)=1$. Since every prime divisor of $\exp(G)/n_1$ is a divisor of $n_0$, it follows from   $\exp(G)/n_1\le 2$ that  $(\exp(G)/n_1)$ divides $ n_0$.
Hence $$\frac{\lcm(\exp(G), n_0)}{n_0n_1}=\frac{\lcm(\exp(G)/n_1, n_0)}{n_0}=1\,.$$
 Thus by Theorem \ref{main1} $\mathcal O_{\mathfrak p}[G]$ is clean.  Since  $p\nmid \exp(G)$, we have $\gcd(n, \exp(G))=\gcd(n_0, \exp(G)/n_1)\le 2$. Thus  it follows from   Proposition \ref{p2} that $\Q(\zeta_{n})[G]$ is $*$-clean and hence by Theorem \ref{t6} $\mathcal O_{\mathfrak p}[G]$ is $*$-clean.

\medskip
2. Suppose  $\mathcal O_{\mathfrak p}[G]$ is $*$-clean.
Let $m'=\frac{\lcm(\exp(G), n_0)}{n_0n_1}$.
Since $\mathcal O_{\mathfrak p}[G]$ is clean, it follows from  Theorem \ref{main1} that  $$\ord_{n_1}p=\varphi(n_1), \quad\quad \ord_{n_0m'}p=m'\ord_{n_0}p,\quad  \text{and} \quad \gcd(\varphi(n_1), m'\ord_{n_0}p)=1\,.$$
 By Theorem \ref{t6} and Proposition \ref{p2}, we have  $\exp(G)\ge 3$ and $\gcd(n,\exp(G))\le 2$. Thus $\gcd(n_0, \exp(G)/n_1)\le 2$. Since every prime divisor of $\exp(G)/n_1$ is a divisor of $n_0$, we obtain $$\exp(G)=2^{\ell}n_1\text{ for some }\ell\in \N_0\,.$$
If  $\ell\ge 2$, then  $n_0=2n_0'$ with $n_0'$ is odd which implies that $m'=2^{\ell-1}$. Thus
$$
2^{\ell-1}\ord_{n_0}p=m'\ord_{n_0}p=\ord_{m'n_0}p=\lcm(\ord_{2^{\ell}}p, \ord_{n_0'}p)\le 2^{\ell-2}\ord_{n_0'}p=2^{\ell-2}\ord_{n_0}p\,,$$ a contradiction.
Thus $\exp(G)\le 2n_1$ and $m'=1$. The assertion follows.
\end{proof}

Next, we provide some ($*$-clean or non $*$-clean) clean group rings in  each case of the characterizations of Theorems \ref{main1} and \ref{main2}.

\begin{example}Let $K=\Q(\zeta_n)$ be a cyclotomic field for some $n\in \N$, $\mathcal O=\Z[\zeta_n]$ its rings of integers, and $G$ a finite abelian group with $\exp(G)\ge 3$.
\begin{enumerate}
\item If $p$ is a primitive root of unity of $\exp(G)$, then $\Z_{(p)}[G]$ is $*$-clean.

\item Suppose $\gcd(\exp(G), n)=1$ and $\exp(G)$ has a primitive root. If there is a prime divisor $q$ of $\varphi(n)$ such that $q\nmid \varphi(\exp(G))$, then there exists $x, y\in \N$ with $\gcd(x,n)=1$ and $\gcd(y, \exp(G))=1$  such that $\ord_nx=q$ and $\ord_{\exp(G)}y=\varphi(\exp(G))$. By Chinese Remainder Theory, there exists $z\in \N$ with $\gcd(z, n\exp(G))=1$ such that $\ord_nz=q$ and $\ord_{\exp(G)}z=\varphi(\exp(G))$. By Dirichlet's prime number theorem, there is a prime $p$ such that $p\equiv z\pmod {n\exp(G)}$. Let $\mathfrak p\subset \mathcal O$ be a prime ideal such that $\mathfrak p\cap \Z=p\Z$.
Then by Theorem \ref{main2} $\mathcal O_{\mathfrak p}[G]$ is $*$-clean.

\item Suppose  $\gcd(\exp(G), n)\ge 3$, $\gcd\big(\frac{\exp(G)}{\gcd(\exp(G),n)}, n\big)=1$, and $\frac{\exp(G)}{\gcd(\exp(G),n)}$ has a primitive root. If there is a prime divisor $q$ of $\varphi(n)$ such that $q\nmid \varphi\big(\frac{\exp(G)}{\gcd(\exp(G),n)}\big)$, then there exists a prime $p$ such that $\gcd(\ord_np, \ord_{\frac{\exp(G)}{\gcd(\exp(G),n)}}p)=1$. Let $\mathfrak p\subset \mathcal O$ be a prime ideal such that $\mathfrak p\cap \Z=p\Z$.  Then by Theorems \ref{main1} and \ref{main2},  $\mathcal O_{\mathfrak p}[G]$ is clean but not $*$-clean.

\item Let $n=7$, $\exp(G)=49\times 3$, and let $\mathfrak p\subset \mathcal O$ be a prime ideal such that $\mathfrak p\cap \Z=23\Z$.
Since $\ord_7{23}=3$, $\ord_{3}{23}=2=\varphi(3)$, and $\ord_{49}{23}=21=7\ord_7{23}$, it follows from    Theorems \ref{main1} and \ref{main2} $\mathcal O_{\mathfrak p}[G]$ is clean but not $*$-clean.

\end{enumerate}
\end{example}

\bigskip
\section{Group rings over local subrings of quadratic fields}
\medskip

In this section, we investigate when a group ring over a local subring of a quadratic field is clean. Let $d$ be a non-zero square-free integer with $d\neq 1$,  $K=\Q(\sqrt{d})$ a quadratic number field,
\begin{align*}
\omega=\left\{\begin{aligned}
&\sqrt{d}      &&\text{ if }d\equiv 2,3\pmod 4\,,\\
&\frac{1+\sqrt{d}}{2} &&\text{ if }d\equiv 1\pmod 4\,,
\end{aligned}
\right.
\text{ and }
\Delta=\left\{\begin{aligned}
&4d      &&\text{ if }d\equiv 2,3\pmod 4\,,\\
&d &&\text{ if }d\equiv 1\pmod 4\,.
\end{aligned}
\right.
\end{align*}
Then $\mathcal O_K=\Z[\omega]$ is the ring of integers  and $\Delta$ is the discriminant of $K$.

 For an odd prime $p$ and an integer $a$, we denote by $\left(\frac{a}{p}\right)\in \{-1,0,1\}$ the Legendre symbol of $a$ modulo $p$.

We first provide two useful lemmas.

\begin{lemma}\label{2.10}
Let $d\neq 1$ be a non-zero square-free integer and let $\Delta$ be the discriminant of $\Q(\sqrt{d})$.
Then $\Q(\sqrt{d})\subset \Q(\zeta_n)$ if and only if $n$ is a multiple of $\Delta$.
\end{lemma}

\begin{proof}
This follows from  \cite[Corollary 4.5.5]{We06}
\end{proof}

\begin{lemma}\label{2.11}
Let $d\neq 1$ be a non-zero square-free integer and let $I$ be a prime ideal of $\mathcal O_K$, where $K=\Q(\sqrt{d})$.
Suppose $\Delta$ is the discriminant of $K$ and $\ch \mathcal O_K/I=p$, where $p$ is a prime.

\begin{enumerate}
\item If $p=2$, then $N(I)=p$ if and only if $\Delta\not\equiv 5\pmod 8$.

\item If $p$ is odd, then $N(I)=p$ if and only if $\big(\frac{\Delta}{p}\big)=1$ or $0$.
\end{enumerate}
\end{lemma}

\begin{proof}
This follows by  \cite[Theorem 22, III.2.1, and V.1.1]{Fr-Ta92}.
\end{proof}

\begin{proof}[{\bf Proof of Theorem \ref{t3}}]
  Let $R=\mathcal O_{\mathfrak p}$. By Theorem \ref{t2}, we have $R[G]=\mathcal O_{\mathfrak p}[G]$ is clean if and only if $[\Q(\zeta_m):\Q(\zeta_m)\cap \Q(\sqrt{d})]=\ord_m( N(\mathfrak p))$ for every divisor $m$ of $n$.

\medskip
1. Since $\Delta\nmid n$,
it follows  by Lemma \ref{2.10} that $\Q(\sqrt{d})\cap\Q(\zeta_{m})=\Q$ for every positive divisor $m$ of $n$.

1.1. Suppose the item 1.(a) or 1.(b) holds.
By Lemma \ref{2.11} we have $ N(\mathfrak p)=p$. Therefore for every divisor $m$ of $n$, we obtain that $p$ is a primitive root of unity of $m$ and hence
$$[\Q(\zeta_m):\Q(\zeta_m)\cap \Q(\sqrt{d})]=[\Q(\zeta_m):\Q]=\varphi(m)=\ord_mp=\ord_m( N(\mathfrak p))\,.$$
   Thus $R[G]$ is clean. 
   
   Suppose the item 1.(c) holds.
   By Lemma \ref{2.11} we have $ N(\mathfrak p)=p^2$. Thus 
   $$[\Q(\zeta_2):\Q(\zeta_2)\cap \Q(\sqrt{d})]=[\Q(\zeta_2):\Q]=\varphi(2)=\ord_2p^2=\ord_2( N(\mathfrak p))\,,$$
   whence $R[G]$ is clean.

1.2. Conversely, suppose $R[G]$ is clean. Then $[\Q(\zeta_n):\Q(\zeta_n)\cap \Q(\sqrt{d})]=[\Q(\zeta_n):\Q]=\ord_n( N(\mathfrak p))$ implies that $\varphi(n)=\ord_n( N(\mathfrak p))$.
 Thus either $ N(\mathfrak p)=p$ is a primitive root of unity of $n$, or $ N(\mathfrak p)=p^2$ and $\ord_np=\varphi(n)$ is odd, i.e., $n\le 2$. The assertions follow by Lemma \ref{2.11}.
 \medskip

 2.1.
 Suppose that $R[G]$ is clean. Since $\Delta=4d\nmid 4$, we have
  $[\Q(\zeta_4):\Q(\zeta_4)\cap \Q(\sqrt{d})]=[\Q(\zeta_4):\Q]=\ord_4( N(\mathfrak p))$. Thus $\varphi(4)=\ord_4( N(\mathfrak p))$ and hence $ N(\mathfrak p)=p\equiv 3\pmod 4$.
   If $p\mid d$, then $p\mid n$, a contradiction. Thus $p\nmid d$ and hence by Lemma \ref{2.11}.2 $\big(\frac{\Delta}{p}\big)=1$.

  Since $[\Q(\zeta_n):\Q(\zeta_n)\cap \Q(\sqrt{d})]=[\Q(\zeta_n):\Q(\sqrt{d})]=\ord_n( N(\mathfrak p))$, we obtain that $\varphi(n)/2=\ord_n( N(\mathfrak p))=\ord_np$. Since $4\mid n$, we may assume that $n=2^{\ell}n'$ with $\ell\ge 2$ and $n'$ is odd. Thus $(\Z/n\Z)^{\times}\cong (\Z/2^{\ell}\Z)^{\times}\times (\Z/n'\Z)^{\times}$. Since $(\Z/n\Z)^{\times}$ has an element of order $\varphi(n)/2$,  we obtain that $n'=1$ if $\ell\ge 3$ and $n'$ is a prime power if $\ell=2$.  Thus  $n=4q^{\ell}$ and $d\t q^{\ell}$, where $q$ is a prime. Note that $d$ is square-free. Therefore $|d|=q$ is a prime.


  \medskip
2.2. Conversely, $\big(\frac{\Delta}{p}\big)=1$ implies that $ N(\mathfrak p)=p$.
Suppose $|d|=2$ and $n=2^{\ell}$ with $\ell\ge 3$. Let $m$ be a positive divisor of $n$. If $m=4$,  then $|\Q(\zeta_4):\Q(\zeta_4)\cap \Q(\sqrt{d})|=2$ and $\ord_4p=2$ by $p\equiv 3\pmod 4$. Thus $[\Q(\zeta_4):\Q(\zeta_4)\cap \Q(\sqrt{d})]=\ord_4p$. If $m=2^t$ with $t\ge 3$, then $[\Q(\zeta_m):\Q(\zeta_m)\cap \Q(\sqrt{d})]=2^{t-2}$ and $\ord_mp=m/4=2^{t-2}$ by $2^{\ell-2}=\varphi(n)/2=\ord_np$.  Thus $[\Q(\zeta_m):\Q(\zeta_m)\cap \Q(\sqrt{d})]=\ord_m( N(\mathfrak p))$.
  Putting all these together, we obtain that $R[G]$ is clean.

  Suppose $|d|\ge 3$ is a prime and $n=4|d|^{\ell}$. Let $m$ be a positive divisor of $n$. If $m=|d|^t$ for some $1\le t\le \ell$, then $4d\nmid m$ and hence $[\Q(\zeta_m):\Q(\zeta_m)\cap \Q(\sqrt{d})]=\varphi(m)=|d|^{t-1}(|d|-1)$. Since $p$ is a primitive root of unity of $|d|^{\ell}$, we obtain $\varphi(m)=\ord_mp$. Therefore  $[\Q(\zeta_m):\Q(\zeta_m)\cap \Q(\sqrt{d})]=\ord_m( N(\mathfrak p))$. If $m=2|d|^t$ for some $1\le t\le \ell$, then $4d\nmid m$ and hence $[\Q(\zeta_m):\Q(\zeta_m)\cap \Q(\sqrt{d})]=\varphi(m)=|d|^{t-1}(|d|-1)$. Since $p$ is a primitive root of unity of $|d|^{\ell}$, we obtain $\varphi(m)=\ord_mp$. Therefore  $[\Q(\zeta_m):\Q(\zeta_m)\cap \Q(\sqrt{d})]=\ord_m(N(\mathfrak p))$. If $m=4|d|^t$ for some $1\le t\le \ell$, then $4d\mid m$ and hence $[\Q(\zeta_m):\Q(\zeta_m)\cap \Q(\sqrt{d})]=\varphi(m)/2=|d|^{t-1}(|d|-1)$. Since $p$ is a primitive root of unity of $|d|^{\ell}$, we obtain $\varphi(m)=\ord_mp$. Therefore  $[\Q(\zeta_m):\Q(\zeta_m)\cap \Q(\sqrt{d})]=\ord_m(N(\mathfrak p))$.
 If $m=4$,  then $[\Q(\zeta_4):\Q(\zeta_4)\cap \Q(\sqrt{2})]=2$ and $\ord_4p=2$ as  $p\equiv 3\pmod 4$. Thus $[\Q(\zeta_4):\Q(\zeta_4)\cap \Q(\sqrt{2})]=\ord_4p$.
    Putting all these together, we obtain that $R[G]$ is clean.

 \medskip
3.
If $p\mid d$, then $p\mid n$, a contradiction. Therefore $\big(\frac{d}{p}\big)=1$ or $-1$.

3.1. Let $R[G]$ be clean. Suppose $|d|$ is a prime and $n=|d|^{\ell}$ or $n=2|d|^{\ell}$ for some $\ell\in \N$.
If $\big(\frac{d}{p}\big)=-1$, then $ N(\mathfrak p)=p^2$ and hence
$[\Q(\zeta(n)):\Q(\zeta(n))\cap \Q(\sqrt{d})]=\ord_n p^2$ implies either $\varphi(n)=\ord_np$ or $\ord_np=\varphi(n)/2$ is odd.
If $\big(\frac{d}{p}\big)=1$, then $ N(\mathfrak p)=p$ and hence $[\Q(\zeta(n)):\Q(\zeta(n))\cap \Q(\sqrt{d})]=\ord_n p$ implies that $\varphi(n)/2=\ord_np$. Putting all these together, we have either $\ord_np=\frac{2\varphi(n)}{3+\big(\frac{d}{p}\big)}$, or $\ord_np=\varphi(n)/2$ is odd with $\big(\frac{d}{p}\big)=-1$. Note that $\varphi(n)/2$ is odd if and only if $|d|\equiv 3\pmod 4$, i.e., $d<0$. Therefore (a) holds.

Otherwise, there exists an $m\mid n $ with $m\ge 3$ such that $d\nmid m$. Therefore
$[\Q(\zeta_m):\Q(\zeta_m)\cap Q(\sqrt{d})]=\varphi(m)=\ord_m( N(\mathfrak p))$. Since $\varphi(m)$ must be even, we obtain that  $ N(\mathfrak p)=p$ and hence $\big(\frac{d}{p}\big)=1$. Since $d\t n$, we have  $\varphi(n)/2=\ord_np$. Therefore $(\Z/n\Z)^{\times}\cong C_{\varphi(n)}$ or $C_2\oplus C_{\varphi(n)/2}$,  which implies that $n=q_1^{\ell_1}q_2^{\ell_2}$ or $2q_1^{\ell_1}q_2^{\ell_2}$ or $4q_1^{\ell_1}$, where $q_1,q_2$ are distinct odd primes with $q_1\t d$, $\ell_1\in \N$ and $\ell_2\in \N_0$.
 Note that $d$ is square-free.
Then either $|d|=q_1$ and $n\in \{q_1^{\ell_1}q_2^{\ell_2}, 2q_1^{\ell_1}q_2^{\ell_2}, 4q_1^{\ell_1}\}$ with $\ell_1,\ell_2\in \N$,    or $|d|=q_1q_2$ and $n\in \{q_1^{\ell_1}q_2^{\ell_2}, 2q_1^{\ell_1}q_2^{\ell_2}\}$ with $\ell_1,\ell_2\in \N$.

Suppose $|d|=q_1$.  If  $n=q_1^{\ell_1}q_2^{\ell_2}$ or $2q_1^{\ell_1}q_2^{\ell_2}$, then $[\Q(\zeta_{q_1^{\ell_1}}):\Q(\zeta_{q_1^{\ell_1}})\cap Q(\sqrt{d})]=\varphi(q_1^{\ell_1})/2=\ord_{q_1^{\ell_1}}p$ and $[\Q(\zeta_{q_1^{\ell_2}}):\Q(\zeta_{q_2^{\ell_2}})\cap Q(\sqrt{d})]=\varphi(q_2^{\ell_2})=\ord_{q_2^{\ell_2}}p$. Since 
\begin{align*}
[\Q(\zeta_{q_1^{\ell_1}q_2^{\ell_2}}):\Q(\zeta_{q_1^{\ell_1}q_2^{\ell_2}})\cap Q(\sqrt{d})]&=\varphi(q_1^{\ell_1}q_2^{\ell_2})/2=\ord_{q_1^{\ell_1}q_2^{\ell_2}}(p)\\
&=\lcm(\ord_{q_1^{\ell_1}}p, \ord_{q_2^{\ell_2}}p)=\frac{\ord_{q_1^{\ell_1}}p \ord_{q_2^{\ell_2}}p}{\gcd(\ord_{q_1^{\ell_1}}p, \ord_{q_2^{\ell_2}}p)}\\
&=\frac{\varphi(q_1^{\ell_1})/2 \cdot \varphi(q_2^{\ell_2})}{\gcd(\varphi(q_1^{\ell_1})/2, \varphi(q_2^{\ell_2}))}\\
&=\frac{\varphi(q_1^{\ell_1}q_2^{\ell_2})/2}{\gcd(q_1^{\ell_1-1}(q_1-1)/2, q_2^{\ell_2-1}(q_2-1))}\,,
\end{align*}
we have $\gcd(q_1^{\ell_1-1}(q_1-1)/2, q_2^{\ell_2-1}(q_2-1))=1$ which implies that $d=-q_1\equiv 1\pmod 4$. Therefore (c) holds.
If $n=4q_1^{\ell_1}$, then $p\neq 2$ and $[\Q(\zeta_{q_1^{\ell_1}}):\Q(\zeta_{q_1^{\ell_1}})\cap Q(\sqrt{d})]=\varphi(q_1^{\ell_1})/2=\ord_{q_1^{\ell_1}}p$, $[\Q(\zeta_{4}):\Q(\zeta_{4})\cap Q(\sqrt{d})]=\varphi(4)=\ord_{4}p$ which implies that $p\equiv 3\pmod 4$.
 Since 
\begin{align*}
[\Q(\zeta_{4q_1^{\ell_1}}):\Q(\zeta_{4q_1^{\ell_1}})\cap Q(\sqrt{d})]&=\varphi(4q_1^{\ell_1})/2=\ord_{4q_1^{\ell_1}}(p)\\
&=\lcm(\ord_{4}p, \ord_{q_1^{\ell_1}}p)=\frac{2\ord_{q_1^{\ell_1}}p}{\gcd(2, \ord_{q_1^{\ell_1}}p)}\\
&=\frac{\varphi(4q_1^{\ell_1})/2}{\gcd(2, q_1^{\ell_1-1}(q_1-1)/2)}\,,
\end{align*}
we have $\gcd(2, q_1^{\ell_1-1}(q_1-1)/2)=1$, whence $q_1\equiv 3\pmod 4$. Therefore (b) holds.

Suppose $|d|=q_1q_2$. Then  $[\Q(\zeta_{q_1^{\ell_1}}):\Q(\zeta_{q_1^{\ell_1}})\cap Q(\sqrt{d})]=\varphi(q_1^{\ell_1})=\ord_{q_1^{\ell_1}}(p)$ and $[\Q(\zeta_{q_2^{\ell_2}}):\Q(\zeta_{q_2^{\ell_2}})\cap Q(\sqrt{d})]=\varphi(q_2^{\ell_2})=\ord_{q_2^{\ell_2}}(p)$, whence $p$ is a primitive root of unity of both $q_1^{\ell_1}$ and $q_2^{\ell_2}$. 
Since 
\begin{align*}
[\Q(\zeta_{q_1^{\ell_1}q_2^{\ell_2}}):\Q(\zeta_{q_1^{\ell_1}q_2^{\ell_2}})\cap Q(\sqrt{d})]&=\varphi(q_1^{\ell_1}q_2^{\ell_2})/2=\ord_{q_1^{\ell_1}q_2^{\ell_2}}(p)\\
&=\lcm(\ord_{q_1^{\ell_1}}p, \ord_{q_2^{\ell_2}}p)=\frac{\ord_{q_1^{\ell_1}}p \ord_{q_2^{\ell_2}}p}{\gcd(\ord_{q_1^{\ell_1}}p, \ord_{q_2^{\ell_2}}p)}\\
&=\frac{\varphi(q_1^{\ell_1}) \varphi(q_2^{\ell_2})}{\gcd(\varphi(q_1^{\ell_1}), \varphi(q_2^{\ell_2}))}\\
&=\frac{\varphi(q_1^{\ell_1}q_2^{\ell_2})}{\gcd(q_1^{\ell_1-1}(q_1-1), q_2^{\ell_2-1}(q_2-1))}\,,
\end{align*}
we have $\gcd(q_1^{\ell_1-1}(q_1-1)/2, q_2^{\ell_2-1}(q_2-1)/2)=1$. Therefore (d) holds.


\medskip
3.2. Conversely,
suppose that (a) holds. Let $m$ be a positive divisor of $n$ with  $m\ge 3$. Then $d\mid m$. If $\big(\frac{d}{p}\big)=1$, then $ N(\mathfrak p)=p$ and hence $[\Q(\zeta_m):\Q(\zeta_m)\cap \Q(\sqrt{d})]=\varphi(m)/2=\ord_mp=\ord_m( N(\mathfrak p))$,  implying that $R[G]$ is clean. If $\big(\frac{d}{p}\big)=-1$ and either $\ord_np=\varphi(n)$ or $\ord_np=\varphi(n)/2$ with $|d|=-d\equiv 3\pmod 4$, then $ N(\mathfrak p)=p^2$, $\ord_np^2=\varphi(n)/2$, and hence $[\Q(\zeta_m):\Q(\zeta_m)\cap \Q(\sqrt{d})]=\varphi(m)/2=\ord_mp^2=\ord_m( N(\mathfrak p))$, implying that $R[G]$ is clean. 

Suppose that (b) holds. Then $ N(\mathfrak p)=p$ and $\ord_{q^i}p=q^{i-1}(q-1)/2$ is odd for all $i\in [1,\ell]$. Let $m$ be a positive divisor of $n$ with  $m\ge 3$.
Then $d\mid m$ and hence
$[\Q(\zeta_m):\Q(\zeta_m)\cap \Q(\sqrt{d})]=\varphi(m)/2=\ord_mp=\ord_m( N(\mathfrak p))$.
Therefore $R[G]$ is clean.

Suppose that (c) holds. Then $ N(\mathfrak p)=p$, $\ord_{q_1^i}p=q_1^{i-1}(q_1-1)/2$ is odd,  $p$ is a primitive root of $q_2^j$, and $\gcd(\ord_{q_1^i}p, \ord_{q_2^i}p)=1$ for all $i\in [1,\ell_1]$  and $j\in [1,\ell_2]$. Let $m$ be a positive divisor of $n$ with  $m\ge 3$.
If $m=q_2^t$ or $2q_2^t$ for some $1\le t\le \ell_2$, then $d\nmid m$ and hence
$[\Q(\zeta_m):\Q(\zeta_m)\cap \Q(\sqrt{d})]=\varphi(m)=\ord_mp=\ord_m( N(\mathfrak p))$.
Otherwise $d\mid m$ and hence  $[\Q(\zeta_m):\Q(\zeta_m)\cap \Q(\sqrt{d})]=\varphi(m)/2=\ord_mp=\ord_m( N(\mathfrak p))$.
Therefore $R[G]$ is clean.

Suppose that (d) holds. Then $ N(\mathfrak p)=p$. Let $m$ be a positive divisor of $n$ with  $m\ge 3$.
If $m=q_1^{t_1}$, or $2q_1^{t_1}$, or $q_2^{t_2}$, or $2q_2^{t_2}$ for some $1\le t_1\le \ell_1$ or some $1\le t_2\le \ell_2$, then $d\nmid m$ and hence
 $[\Q(\zeta_m):\Q(\zeta_m)\cap \Q(\sqrt{d})]=\varphi(m)=\ord_mp=\ord_m( N(\mathfrak p))$.
 If $m=q_1^{t_1}q_2^{t_2}$ or $2q_1^{t_1}q_2^{t_2}$, then $d\mid m$ and hence  $[\Q(\zeta_m):\Q(\zeta_m)\cap \Q(\sqrt{d})]=\varphi(m)/2=\ord_mp=\ord_m( N(\mathfrak p))$.
 Therefore $R[G]$ is clean. 
\end{proof}

Next we characterize when such a group ring is $*$-clean. We first prove the following lemma.

\begin{lemma}\label{l22}
Let $d\neq 1$ be a non-zero square free integer.
Then $\Q(\sqrt{d})(\zeta_n+\zeta_n^{-1})=\Q(\sqrt{d})(\zeta_n)$ if and only if  either  \big( $d<0$ and $\Delta\mid n$ \big)   or $n\le 2$, where $n\in \N$ and $\Delta$ is the discriminant of $\Q(\sqrt{d})$.
\end{lemma}

\begin{proof}
 If $n\le 2$, it is obvious that $\Q(\sqrt{d})(\zeta_n+\zeta_n^{-1})=\Q(\sqrt{d})(\zeta_n)$. Now we let $n\ge 3$.

 Suppose that  $d<0$ and $\Delta\mid n$. Then by Lemma \ref{2.10} $\Q(\sqrt{d})\subset \Q(\zeta_n)$ and hence $\Q(\sqrt{d})(\zeta_n+\zeta_n^{-1})\subset\Q(\zeta_n)$. Since $n\ge 3$ and $d<0$, we have $$[\Q(\zeta_n):\Q(\zeta_n+\zeta_n^{-1})]=[\Q(\sqrt{d})(\zeta_n+\zeta_n^{-1}):\Q(\zeta_n+\zeta_n^{-1})]=2\,.$$ Therefore $\Q(\sqrt{d})(\zeta_n+\zeta_n^{-1})=\Q(\zeta_n)=\Q(\sqrt{d})(\zeta_n)$.

Suppose that  $d>0$. Then $\Q(\sqrt{d})(\zeta_n+\zeta_n^{-1})\subset \R$ and by $n\ge 3$, we have $\Q(\sqrt{d})(\zeta_n)\not\subset \R$. Hence $\Q(\sqrt{d})(\zeta_n+\zeta_n^{-1})\neq\Q(\sqrt{d})(\zeta_n)$.

Suppose that  $d<0$ and $\Delta\nmid n$. Thus by Lemma \ref{2.10} $\sqrt{d}\not\in \Q(\zeta_n)$. Therefore
$[\Q(\sqrt{d})(\zeta_n):\Q]=2\varphi(n)$ and $[\Q(\sqrt{d})(\zeta_n+\zeta_n^{-1}):\Q]=\varphi(n)$. It follows that $\Q(\sqrt{d})(\zeta_n+\zeta_n^{-1})\neq\Q(\sqrt{d})(\zeta_n)$.
\end{proof}

\begin{proposition}\label{p3}
Let $G$ be a finite abelian group with exponent $n$. Then $\Q(\sqrt{d})[G]$ is $*$-clean if and only if $n\ge 3$ and  either  $d>0$ or  $\Delta\nmid n$, where $\Delta$ is the discriminant of $\Q(\sqrt{d})$.
\end{proposition}

\begin{proof}
This result follows from    \cite[Theorem 1.2]{Ha-Re-Zh17a} and Lemma \ref{l22}.
\end{proof}

\begin{theorem}\label{t4}
Let $K=\Q(\sqrt{d})$ be a quadratic field for some non-zero square-free integer $d\neq 1$, $\mathcal O$ its ring of integers, $\mathfrak p\subset \mathcal O$ a nonzero prime ideal with $p\Z=\mathfrak p\cap \mathcal O$, and $G$ a finite abelian group with $p\nmid \exp(G)$.  Let $\Delta$ be the  discriminant of the field extension $K/\Q$. Then
 \begin{enumerate}
 \item if $d>0$, then $\mathcal O_{\mathfrak p}[G]$ is $*$-clean if and only $\mathcal O_{\mathfrak p}[G]$ is clean and $\exp(G)\ge 3$.

 \item if $d<0$, then $\mathcal O_{\mathfrak p}[G]$ is $*$-clean if and only if
  $\Delta\nmid \exp(G)$, $p$ is a primitive root of unity of $\exp(G)$, $\exp(G)\ge 3$, and $\big(\frac{\Delta}{p}\big)=1$ or $0$.
  \end{enumerate}
\end{theorem}

\begin{proof}
1. Let $d>0$.
Suppose $\mathcal O_{\mathfrak p}[G]$ is clean and $\exp(G)\ge 3$. Then by Proposition \ref{p3} $\Q(\sqrt{d})[G]$ is $*$-clean. It follows from    Theorem \ref{t6} that $\mathcal O_{\mathfrak p}[G]$ is $*$-clean.

Conversely, suppose $\mathcal O_{\mathfrak p}[G]$ is $*$-clean. Then by Theorem \ref{t6} $\Q(\sqrt{d})[G]$ is $*$-clean. It follows from    Proposition \ref{p3} that $\exp(G)\ge 3$.

2. Let $d<0$. Suppose that $\Delta\nmid \exp(G)$, $p$ is a primitive root of unity of $\exp(G)$, $\exp(G)\ge 3$, and $\big(\frac{\Delta}{p}\big)=1$ or $0$. Then by Theorem \ref{t3}
$\mathcal O_{\mathfrak p}$ is clean and
by Proposition \ref{p3} $\Q(\sqrt{d})[G]$ is $*$-clean. It follows from   Theorem \ref{t6} that $\mathcal O_{\mathfrak p}[G]$ is $*$-clean.

Conversely,
suppose $\mathcal O_{\mathfrak p}[G]$ is $*$-clean. Then by Theorem \ref{t6} $\Q(\sqrt{d})[G]$ is $*$-clean and hence by Proposition \ref{p3} $\exp(G)\ge 3$ and $\Delta\nmid \exp(G)$. It follows from    Theorem \ref{t3}.1 that $p$ is a primitive root of unity of $\exp(G)$ and $\big(\frac{\Delta}{p}\big)=1$ or $0$.
\end{proof}

We close the paper with the following examples which provide some ($*$-clean or non $*$-clean) clean group rings for each case of the characterizations of Theorems \ref{t3} and \ref{t4}.

\begin{example}
\begin{enumerate}
 \item Let $\mathcal O$ be the ring of integers of $\Q(\sqrt{d})$ and let $G$ be a finite abelian group with $\gcd(\exp(G), d)=1$, where $d\neq 1$ is a square free integer and $\exp(G)\neq 4$ has a primitive root. Suppose that  $d=\delta d_0$ such that $d_0$ is the maximal odd positive  divisor of $d$. Thus $\delta\in \{-1,2,-2\}$. For every prime $p$ with $p\equiv 1\pmod {8d_0}$, we have $\big(\frac{d}{p}\big)=\big(\frac{d_0}{p}\big)=1$. Since there exists $x\in \N$ with $\gcd(x, \exp(G))=1$ such that $\ord_{\exp(G)}x=\varphi(\exp(G))$, for every  prime $p$ with $p\equiv x\pmod {\exp(G)}$, we have $\ord_{\exp(G)}p=\varphi(\exp(G))$. Note that $\mathsf v_2(\exp(G))\le 1$. By Dirichlet's prime number theorem, there is a prime $p$ such that $p\equiv 1\pmod {8d_0}$ and $p\equiv x\pmod {\exp(G)}$. Let $\mathfrak p\subset \mathcal O$ be a prime ideal such that $\mathfrak p\cap \Z=p\Z$.
 Then by Theorem \ref{t3}.1 $\mathcal O_{\mathfrak p}[G]$ is clean. If $\exp(G)\ge 3$, then by Theorem \ref{t4} $\mathcal O_{\mathfrak p}[G]$ is $*$-clean.

\item Let $\mathcal O$ be the ring of integers of $\Q(\sqrt{-2})$, let $\mathfrak p\subset \mathcal O$ be a prime ideal with $\mathfrak p\cap \Z=3\Z$, and let $G$ be a finite abelian group with $\exp(G)=8$.
Then Theorem \ref{t3}.2 and Theorem \ref{t4}.2 imply that $\mathcal O_{\mathfrak p}[G]$ is clean but not $*$-clean.

\item Let $\mathcal O$ be the ring of integers of $\Q(\sqrt{3})$, let $\mathfrak p\subset \mathcal O$ be a prime ideal with $\mathfrak p\cap \Z=11\Z$, and let $G$ be a finite abelian group with $\exp(G)=12$.
Then Theorem \ref{t3}.2 and Theorem \ref{t4}.1 imply that $\mathcal O_{\mathfrak p}[G]$ is clean as well as $*$-clean.

\item Let $\mathcal O$ be the ring of integers of $\Q(\sqrt{5})$, let $\mathfrak p\subset \mathcal O$ be a prime ideal with $\mathfrak p\cap \Z=19\Z$, and let $G$ be a finite abelian group with $\exp(G)=5$.
Then Theorem \ref{t3}.3.a and Theorem \ref{t4}.1 imply that $\mathcal O_{\mathfrak p}[G]$ is clean as well as $*$-clean.

\item Let $\mathcal O$ be the ring of integers of $\Q(\sqrt{-3})$, let $\mathfrak p\subset \mathcal O$ be a prime ideal with $\mathfrak p\cap \Z=5\Z$, and let $G$ be a finite abelian group with $\exp(G)=6$.
Then Theorem \ref{t3}.3.a and Theorem \ref{t4}.2 imply that $\mathcal O_{\mathfrak p}[G]$ is clean but not $*$-clean.

\item Let $\mathcal O$ be the ring of integers of $\Q(\sqrt{33})$, let $\mathfrak p\subset \mathcal O$ be a prime ideal with $\mathfrak p\cap \Z=2\Z$, and let $G$ be a finite abelian group with $\exp(G)=33$.
Then Theorem \ref{t3}.3.d and Theorem \ref{t4}.1 imply that $\mathcal O_{\mathfrak p}[G]$ is clean as well as $*$-clean.

\end{enumerate}

\end{example}


\providecommand{\bysame}{\leavevmode\hbox to3em{\hrulefill}\thinspace}
\providecommand{\MR}{\relax\ifhmode\unskip\space\fi MR }
\providecommand{\MRhref}[2]{%
  \href{http://www.ams.org/mathscinet-getitem?mr=#1}{#2}
}
\providecommand{\href}[2]{#2}

\end{document}